\documentclass[10pt]{article}
\usepackage[textwidth=110mm, textheight=181mm]{geometry}

\usepackage{color}
\usepackage[dvipdfmx]{graphicx}
\usepackage{eufrak}
\usepackage{float}
\usepackage{etex}
\usepackage[all]{xy}
\usepackage{amssymb, amsmath}
\usepackage{amsthm}

\numberwithin{equation}{section}

\usepackage{comment} 

\usepackage{tabularx} %

\newtheorem{thm}{Theorem}[section]
 
\newtheorem{lem}[thm]{Lemma} 
\newtheorem{prop}[thm]{Proposition} 
\newtheorem{lemma}[thm]{Lemma}

\theoremstyle{definition}
\newtheorem{defin}[thm]{Definition} 
\newtheorem{rem}[thm]{Remark}

\newcommand{\textR}[1]{\textcolor{red}{#1}}

\newcommand{\A}{\mathcal{A}}
\newcommand{\K}{\mathcal{K}}

\newcommand{\Gcal}{\mathcal{G}}

\newcommand{\R}{\mathbb{R}}

\def\qed{\hfill $\Box$}

\def\bv{\mbox{\boldmath $v$}}

\def\b0{\mbox{\boldmath $0$}}

\allowdisplaybreaks

\newcommand{\addresslist}
{
\begin{tabular}{l}
(M. Hasegawa)\\
Department of Information Science,\\
Center for Liberal Arts and Sciences,\\
Iwate Medical University, \\
2-1-1 Idaidori, Yahaba-cho, Shiwa-gun, Iwate,\\
028-3694, Japan\\
E-mail: {\tt mhaseO\!\!\!aiwate-med.ac.jp}\\[3mm]
(Y. Kabata)\\
School of Information and Data Sciences,\\
Nagasaki University, \\
Bunkyocho 1-14, Nagasaki, 852-8131, Japan\\
{\tt kabata@nagasaki-u.ac.jp}\\[3mm]
(K. Saji)\\
Department of Mathematics,\\
Graduate School of Science, \\
Kobe University, \\
Rokkodai 1-1, Nada, Kobe, 657-8501, Japan\\
{\tt saji@math.kobe-u.ac.jp}
\end{tabular}
}

\begin{document}
\title{Contact cylindrical surfaces and a projection
of a surface around a parabolic point}
\author{Masaru Hasegawa, Yutaro Kabata and Kentaro Saji}
\date{\today}
\maketitle
\begin{abstract}
We investigate differential geometric
properties of a parabolic point of a surface
in the Euclidean three space.
We introduce the contact cylindrical surface which
is a cylindrical surface having a degenerate contact type with the original surface at a parabolic point.
Furthermore, 
we show that such a contact property gives a characterization
to the $\A$-singularity of the orthogonal projection of a surface from the asymptotic direction.
\end{abstract}
\providecommand{\mathsubcl}[1]
{
  \small	
  \textbf{{2020 Mathematics Subject Classification---}} #1
}

\providecommand{\keywords}[1]
{
  \small	
  \textbf{{Keywords--- }} #1
}

\mathsubcl{
Primary: 57R45, 
Secondary: 53A05,  
58Kxx  
}

\



\section{Introduction}
The general theory to study the contact between two manifolds from the viewpoint of the singularity theory was established by Mather and Montaldi (cf. \cite[Chapter 4]{IRFT}). 
Thanks to this theory,
we can investigate local differential geometry of a surface with respect to the contacts with model manifolds.
Planes, spheres, and cylinders are only three types of homogeneous surfaces in $\R^3$, and have been used as the model manifolds to investigate differential geometry of a surface in  $\R^3$ \cite{FHN, IRFT, Montaldi1986}.
Recently, the contacts of a surface with non-homogeneous surfaces (such as cones or parabolas) are considered
in order to investigate new differential geometric features of a surface \cite{shuheihonda, MN}.
As a new candidate for a model surface, the present paper proposes a cylindrical surface,
which is non-homogeneous and useful to analyze the detailed geometry of a surface around a parabolic point. 

In local differential geometry of a surface in the Euclidean three space $\R^3$,
the Gaussian curvature and its zero set are one of the central objects.
A point where the Gaussian curvature vanishes is usually called
a {\it parabolic point}, and its collection is called a {\it parabolic set}.
In this paper, in order to study local differential geometry of a surface in $\R^3$,
we introduce the notion of a {\it contact cylindrical surface},
which is a cylindrical surface having a very degenerate contact with a surface at a parabolic point.
Considering contact cylindrical surfaces,
we get a new stratification of a jet space of parabolic surface germs as in Theorem \ref{thm:main}.
We emphasize that
the notion of contact cylindrical surface can give 
new differential geometrical invariants of surfaces at parabolic points,
as the contact spheres of a surface gives the principal curvatures of the surface (see Remark \ref{reminv}).
Furthermore, the notion of contact cylindrical surface gives
a suitable characterization to $\A$-singularities
of orthogonal projections, as shown below.

In addition to the Gaussian curvature,
several geometric functions or maps on a surface 
show interesting behaviors around a parabolic point.
Generically, 
the Gauss map has fold singularities
at almost all parabolic points,
and cusp singularities at discrete points on a parabolic set.
The later points are called {\it cusps of Gauss},
and have been investigated in many literatures 
(cf. \cite{cuspgauss, BT2019, FH2012, FHN, IRFT, kabata}).
Such points are also characterized by singularities of height functions,
and more detailed features of parabolic points can be given
by studying singularities of projections of surfaces.

The classification of singularities appearing 
in the orthogonal projections of surfaces
has been well studied in many literatures 
(cf. \cite{cuspgauss, BT2019, shuheihonda, IRFT, kabata, whitney})
in the context of $\A$-equivalence.
On the other hand,
it is still not clear
how such singularities of the orthogonal projections are characterized
in other geometrical terms.
For example, 
for a generic surface,
a {\it gulls singularity} of the orthogonal projection
corresponds to a cusp of Gauss
\cite{cuspgauss,  BT2019}.
However, this is not exact in general.
Especially,
whether the orthogonal projection
has a gulls singularity is determined by the $5$-jet of the surface-germ;
while whether the point is a cusp of Gauss is determined by just the
$4$-jet of the surface-germ.
 
The difficulty of characterizing singularities of orthogonal 
projections by other geometry
is due to the complexity of the classification of map-germs by the $\A$-equivalence
as seen in \cite{Bruce1984, Gaffney1983, IRFT, kabata}.
In this sense, the notion of cusp of Gauss is rather simple,
because it is characterized through geometry of height 
functions under $\K$ equivalence.
As mentioned above, studying surfaces with some model submanifolds 
(such as planes, spheres and cylinders)
through the $\K$ equivalence of the contact maps
is a strong tool of singularity theory \cite{FHN, shuheihonda, IRFT, MN, Montaldi1986}.
Our contact cylindrical surface plays an important role in characterizing $\A$-singularities of orthogonal 
projections
as in Table \ref{cor-comp-all}.
In particular,
all $\A$-singularities appearing in the orthogonal 
projections of projection generic surfaces at parabolic points
are completely characterized from the viewpoint of $\K$-singularities 
of contact functions (see Theorem \ref{characterizationthm} and Remark \ref{genericrem}).

The rest of the paper is organized as follows.
In \S2, we give preliminaries about basic notion of singularity theory
used in this paper.
In \S3, 
 we introduce the notion of a contact cylindrical surface,
 and give a stratification of the jet space of surface-germs at non-umbilical
parabolic points
 induced from it.
 In \S4,
 we give stratifications of
 the jet space of surface-germs at non-umbilical parabolic points
 induced from $\A$-singularities of orthogonal projections,
 and compare it with the stratification given in \S3.

Unless noted otherwise, we consider jets of map-germs and function-germs at the origin throughout this paper.

\section{Preliminaries}
In this section, we introduce basic notions of singularity theory of maps.
Let $(X_i,Y_i)$ $(i=1,2)$ be two pairs of 
submanifold germs at the origin of $\R^n$. 
The pairs are said to {\it have the same contact type} if there exists 
a diffeomorphism-germ $\phi:\R^n,0\to\R^n,0$ such that 
$\phi(X_1)=X_2$ and $\phi(Y_1)=Y_2$.
The following theorem shows that this contact
is measured by $\K$-equivalence.
Two map-germs $f_{i}:\R^2,0\to\R^p,0$ $(i=1,2)$ are {\it $\K$-equivalent\/} 
($f_1\sim_{\K} f_2$) if there exist a diffeomorphism-germ $\phi:\R^2,0\to\R^2,0$ 
and a matrix valued map $M:\R^2,0\to GL(p,\R)$ such that $M f_1=f_2\circ \phi.$
\begin{thm}[\cite{Montaldi1986}]
\label{thm:contact}
Let $h_i:X_i,x_i\to\R^n,0$ $(i=1,2)$ be immersion-germs and $g_i:\R^n,0\to\R^p,0$ be submersion-germs with $Y_i=g_i^{-1}(0)$. Then the pairs $(X_1,Y_1)$ and $(X_2,Y_2)$ have the same contact type if and only if $g_1 \circ h_1$ and $g_2 \circ h_2$ are $\mathcal K$-equivalent. 
\end{thm}
With this theorem, we investigate the surface germ $S$ in terms of the contact
with cylindrical surfaces, using $\K$-equivalence. 
Two map-germs $g_{i}:\R^2,0\to\R^2,0$ $(i=1,2)$ are {\it $\A$-equivalent\/} 
($g_1\sim_{\A} g_2$) if there exist diffeomorphism-germs $\phi_i:\R^2,0\to\R^2,0$ 
$(i=1,2)$ such that $\phi_2\circ g_1=g_2\circ \phi_1.$

The $\A$-class (respectively, the ${\mathcal K}$-class) 
of a given germ 
is called its {\it $\A$-singularity} 
(respectively, ${\mathcal K}$-{\it singularity}).
In this paper, we deal with the
$\A$-singularities with $\A$-codimension $\le5$ whose sets of singular points are singular at $0$
as in Table \ref{Riegerlist}
(see \cite{Rieger1987} for detail).
If a map-germ $g:\R^2,0\to\R^2,0$ is
$\A$-equivalent to one of the germs in the ``Normal form" column in 
Table \ref{Riegerlist}, then $g$ is called the name in
the ``Class name'' column of the table.
\begin{table}[htbp]
\centering
\begin{tabular}{c| l | l }
$\A$-codimension  &Class name & Normal form\\
\hline
\hline
3&beaks($-$), lips($+$) &($x, y^3\pm x^2y)$\\
\hline
4&goose &$(x, y^3+ x^3y)$\\
&gulls& $(x,xy^2+y^4+y^5)$ \\
\hline
5&$\pm$-ugly goose& $(x,y^3 \pm x^4y)$ \\
&ugly gulls& $(x,xy^2+y^4+y^7)$\\
&$12$-singularity & $(x,xy^2+y^5+y^6)$\\
&$16^\pm$-singularity  & $(x,x^2y+y^4\pm y^5)$\\
\end{tabular}
\caption{$\A$-singularities of map-germs $\R^2,0 \to \R^2,0$.}
\label{Riegerlist}
\end{table}

Two $k$-jets $j^kg_{i}(0)$ of map-germs $g_{i}:\R^2,0\to\R^2,0$ $(i=1,2)$ are 
{\it $\Gcal^k$-equivalent\/} ($j^kg_1(0)\sim_{\Gcal^k} j^kg_2(0)$) 
if there exits an action $h \in\Gcal$
so that $j^k(h.g_1)(0)=j^kg_2(0)$, where $\Gcal=\A$ or $\K$.
The following $\A$ or $\K$-determinacy of singularities
allows us to handle each singularity by its finite jet.
\begin{defin}
A map-germ $g:\R^n,0\to\R^p,0$ is said to be $k$-$\Gcal$-\textit{determined} if any 
map-germ $h$ satisfying $j^k h(0)=j^k g(0)$ is $\Gcal$-equivalent to $g$, 
where $\Gcal=\A$ or $\K$. 
The minimum integer $k$ satisfying this property is 
called the \textit{degree of determinacy of} $g$.
\end{defin}
It is known that
each singularity in Table \ref{Riegerlist}
is $k$-$\A$-determined, where $k$ is the highest degree of
the monomials appearing in its normal form.
See \cite[Section 3.2]{Rieger1987}, for detail and proofs.

On the other hand, we also deal with simple
$\mathcal{K}$-singularities of functions with 
two variables, which are classified by Arnol'd (\cite{Arnold1978})
as in Table \ref{tab:ksimple}. 
Each singularity in Table \ref{tab:ksimple}
is $k$-$\K$-determined, where $k$ is the highest degree of
the monomials in its normal form.
\begin{table}[!htb]
\begin{center}
\begin{tabular}{l|l}
\hline
Class name & Normal form\\
\hline
\hline
$A_{k}^{\pm}$ $(k\geq1)$ & $x^2 \pm y^{k+1}$ \\[3pt]
$D_{k}^{\pm}$ $(k\geq4)$& $x^2 y \pm y^{k-1}$ \\[3pt]
$E_6$ & $x^3 + y^4$\\[3pt]
$E_7$ & $x^3 + x y^3$\\[3pt]
$E_8$ & $x^3 + y^5$\\
\hline
\end{tabular}
\end{center}
\caption{Simple $\K$-singularities of smooth functions.}
\label{tab:ksimple}
\end{table}

Unless noted otherwise, we consider jets of map-germs and function-germs
 at the origin,
namely, $j^r\phi=j^r\phi(0)$ for a map $\phi$.

\section{Contact with cylindrical surface along\\ asymptotic direction}

In this section, we consider the contact of a surface-germ $S$ in Euclidean three space at a non-umbilical parabolic point 
with a family of cylindrical surface-germs. 
We call such a surface germ $S$ a {\it parabolic surface-germ}.
Let $S^2=\{X\in\R^3\,|\,|X|=1\}$ be the unit sphere.
For $\bv\in S^2$, let $\bv^\perp$ be the orthogonal complement of $\bv$ in $\R^3$, 
and let $\Gamma \colon \R,0\to \bv^\perp$ be a smooth plane curve. 
The cylindrical-surface germ $CS_{\bv,\Gamma}$ along the cylindrical direction $\bv$ (a line along $\bv$ is called the generatrix) with the base curve $\Gamma$ (known as the directrix) is defined by
$$
CS_{\bv,\Gamma}\colon\R^2,0\to\R^3,0,\quad 
(x,y)\mapsto \Gamma(y)+x\bv.
$$
When we consider $S$ at its parabolic point, which is not umbilical,
without loss of generality, we may assume that $S$ is given by a Monge form 
\begin{equation}\label{monge}
z=f(x,y)=\frac{a_{02}}{2}y^2+\sum_{i+j\ge3}^k \frac{a_{ij}}{i!j!} x^iy^j +O(x,y)^{k+1}\quad (a_{02}\ne0),
\end{equation}
where $O(x,y)^n$  consists of the terms whose degrees are greater than or equal to $n$.

Let $\bv=(1,0,0)$ be an asymptotic direction of $S$ at the origin, 
and let $\Gamma\colon \R,0 \to \R^3, y\mapsto (0, y, \gamma(y))$ be a curve-germ in the $yz$-plane. We shall measure the contact of $S$ with $CS_{\bv,\Gamma}$ expressed by $z=\gamma(y)$ whose cylindrical direction coincides with the unique asymptotic direction of $S$ at the parabolic point.

\begin{rem}
Since a homogeneous surface in $\R^3$ is a plane, a sphere or a right circular cylinder,
the cylindrical surfaces are not homogeneous.
For this reason, our study on the contact of $S$ with $CS_{\bv,\Gamma}$ analyzes
the contact at a special point $\Gamma(0)=(0,0,0)$ coinciding with 
a parabolic point of $S$.
See \cite{FHN, IRFT, Montaldi1986} for the studies with respect to homogeneous cases.
\end{rem}

By Theorem \ref{thm:contact}, the contact of $S$ with $CS_{\bv,\Gamma}$ is measured by $\mathcal{K}$-singularities of the contact map defined by $F(x,y)=f(x,y)-\gamma(y)$. 
The map $\gamma(y)$ can be written as
\begin{equation}\label{eq:gamma}
\gamma(y) = \gamma_1 y+\dfrac12\gamma_2y^2+\dfrac16\gamma_3y^3+O(y)^4,
\end{equation}
where $O(y)^n$ consists of the terms whose degrees are greater than or equal to $n$ and $\gamma_i = d^{(i)}\gamma/dy^{(i)}(0)$, and thus the contact map can be expressed as
\begin{align}\nonumber
F(x,y) & = 
-\gamma_1 y+
\dfrac{a_{02}-\gamma_2}2 y^2 \\
\label{contactfcn}
&\quad + \left(\dfrac{a_{30}}6x^3 + \dfrac{a_{21}}2 x^2 y + \dfrac{a_{12}}2 x y^2 + \dfrac{a_{03} - \gamma_3}6 y^3\right) + O(x,y)^4. 
\end{align}
We use the following proposition to determine the $\K$-singularities of given function germs $g:\R^2,0 \to \R,0$.
\begin{prop}
\label{prop:coordinate}
\begin{enumerate}
\item
If $j^kg=y^2+\sum_{i+j=k}(a_{ij}/i!j!)x^iy^j$ for $k\ge3$. Then $j^k g\sim_{\K^k} y^2+(a_{k0}/k!)x^k$, and thus $g$ has an $A_{k-1}^{\pm}$-singularity when $a_{k0}\not=0$.

\item
If $j^kg=x^2y+\sum_{i+j=k}(a_{ij}/i!j!)x^iy^j$ for $k\ge4$. Then $j^kg\sim_{\K^k} x^2y+(a_{0k}/k!)y^k$, and thus $g$ has a $D_{k+1}^{\pm}$-singularity when $a_{0k}\not=0$.
\item
If $j^kg=x^3+\sum_{i+j=k}(a_{ij}/i!j!)x^iy^j$ for $k\ge4$. Then $j^kg\sim_{\K^k}  x^3+(a_{1\,k-1}/(k-1)!)xy^{k-1}+(a_{0k}/k!)y^{k}$.
\end{enumerate}
\end{prop}

\begin{proof}
Each of the claims (1), (2), (3) follows from the coordinate change
\begin{align}
\label{eq:changeAk}
(x,y)&\mapsto\left(x,y-\frac12\left(\sum_{i+j=k,\,j\ge1}\frac{a_{ij}}{i!j!}x^iy^{j-1}\right)\right),\\
\label{eq:changeD4}
(x,y)&\mapsto\left(x-\frac12\left(\sum_{i+j=k,\, i\ge1,\,j\ge1}\frac{a_{ij}}{i!j!}x^{i-1}y^{j-1}\right), y-\frac{a_{k0}}{k!}x^{k-2}\right).
\\
(x,y)&\mapsto\left(x-\frac13\left(\sum_{i+j=k,\,i\ge2}\frac{a_{ij}}{i!j!}x^{i-2}y^{j}\right), y \right),
\end{align}
respectively.
\end{proof}

We state our main theorem in this section giving a stratification of the space of parabolic surface-germs $S$ with respect to the $\K$-singularities of the contact maps $F$. 

\begin{thm}\label{thm:main}
If a parabolic surface germ $S$ of the form $($\ref{monge}\,$)$ satisfies one of the conditions in the ``Conditions" column in Table \ref{mainthmtable1}, then there exists 
$CS_{\bv,\Gamma}$ such that $F$ of the form \eqref{contactfcn} has one of the simple $\K$-singularities $($note that $E_*$ is not simple$)$ at $0$ in the corresponding item in the ``$\K$-sing. of $F$'' column of Table~ \ref{mainthmtable1}. 
\end{thm}
\begin{table}[!htb]
{\footnotesize 
\begin{tabularx}{\textwidth}{llllc}
\hline
No. & Name & Conditions & $\K$-sing. of $F$ & Cod.\\
\hline
$\text{}^{\rule{0pt}{8pt}}$
 (i)  & $(A_2,D_4|D_{\ge 5})$ & $a_{30}\not=0, Q_{D_4}>0$ & $A_2, D_{\ge 4}$ & 1 \\[2pt]
\hline
$\text{}^{\rule{0pt}{8pt}}$
(ii) & $(A_2,D_4^{+})$  &$a_{30}\not=0, Q_{D_4}<0$ & $A_2, D_4^+$ & 1 \\[2pt]
\hline
$\text{}^{\rule{0pt}{8pt}}$
(iii) &$(A_2, D_4^{+}|E_{6}|E_{7})$  &$a_{30}\not=0, Q_{D_4}=0,$ & $A_2, D_4^+,$ & 2 \\[2pt]
&&$Q_{E_7}\not=0$ &$E_6, E_7$ &  \\[2pt]
\hline
$\text{}^{\rule{0pt}{8pt}}$
(iv) &$(A_3|A_4, D_4|D_{\ge5})$ &$a_{30}=0, a_{21}a_{40}\not=0,$ & $A_3, A_4,$ & 2 \\[2pt]
&&$R_{A_4}\not=0$ &$D_{\ge 4}$ &  \\[2pt]
\hline
$\text{}^{\rule{0pt}{8pt}}$
(v) &$(A_2, D_4^{+}|E_{6}|E_{8}|E_*)$ &$a_{30}\not=0$, $Q_{D_4}=0$, & $A_2, D_4^+, E_6,$ & 3 \\[2pt]
&&$Q_{E_7}=0, Q_{E_*}\not=0$ &$E_8, E_*$&  \\[2pt]
\hline
$\text{}^{\rule{0pt}{8pt}}$
(vi) &$(A_3|A_5|A_6, D_4|D_{\ge5})$ &$a_{30}=0, a_{21}a_{40}\not=0$, & $A_3, A_5, A_6,$ & 3 \\[2pt]
&& $R_{A_4}=0, R_{A_6}\not=0$ &$ D_{\ge 4}$& \\[2pt]
\hline
$\text{}^{\rule{0pt}{8pt}}$
(vii) &$(A_3^-, D_4|D_{\ge5})$ &$a_{30}=a_{40}=0, a_{21}\not=0$ & $A_3^-, D_{\ge 4}$ & 3 \\[2pt]
\hline
$\text{}^{\rule{0pt}{8pt}}$
(viii) &$(A_3, D_{5})$ &$a_{30}=a_{21}=0, a_{12}a_{40}\not=0$& $A_3, D_{5}$ & 3 \\\hline
\end{tabularx}
\label{mainthmtable1}
\caption{The stratification of the jet space of parabolic surface germs.}
}
\end{table}

Here, $E_*$ means that $j^5F\sim_{\K^5}x^3+xy^4$.
The symbols, $Q_{D_4}$, $Q_{E_7}$, $Q_{E_*}$, $R_{A_4}$ and $R_{A_6}$ 
are constant terms consisting of coefficients $a_{ij}$ of \eqref{monge}, and 
given by
\eqref{eq:qd4def},
\eqref{eq:qe7def},
\eqref{eq:qestdef},
\eqref{eq:ra4def} and
\eqref{eq:ra6def} respectively.

\begin{proof}[Proof of Theorem \ref{thm:main}]
We see immediately
$F$ is regular at $0$ if and only if $\gamma_1\not=0$,
and
$F$ has an $A_{\geq 2}$-singularity at $0$ if and only if $\gamma_1=0$ and $\gamma_2\not=a_{02}$.
Moreover, 
$F$ has a degenerate critical point at $0$ 
if and only if $\gamma_1=0$ and $\gamma_2=a_{02}$.

Therefore, we should investigate 
the stratification for higher order terms of $f$ and $\gamma$ in both cases $\gamma_2\not=a_{02}$ and $\gamma_2=a_{02}$. 
We will introduce some lemmas for proving Theorem~\ref{thm:main}.
Note that whether $a_{30}\not=0$ is important in our stratification,
thus we divide our proofs into two steps.
Strata of no. (i)-(iii) and (v) are given from Lemmas \ref{lem:A2}-\ref{lem:strata3_5},
where $a_{30}\not=0$ is assumed;
while strata of no. (iv) and (vi)-(viii) are 
given from Lemmas \ref{lem:strata6_7-1}-\ref{proDEonridge},
where $a_{30}=0$ is assumed.

In the following, we assume $\gamma_1=0$. First, we deal with the case $a_{30}\not=0$. By \eqref{contactfcn} and (1) of Proposition~\ref{prop:coordinate}, we easily show the following lemma.  
\begin{lem}\label{lem:A2}
If $\gamma_2 \ne a_{02}$ and $a_{30} \ne 0$, then $F$ has an 
$A_2$-singularity at $0$.
\end{lem}

The cubic discriminant $DC_f(\gamma_3)$ of the cubic part of \eqref{contactfcn}
is given by
\begin{align}
\label{eq:disc}
DC_{f}(\gamma_3)&:=\frac{1}{48}\left[-a_{30}^2 \gamma_3^2+\left(2 a_{30}^2 a_{03} - 6 a_{30} a_{21} a_{12} + 4 a_{21}^3\right)\gamma_3\right. \\ 
&\left.- (a_{30}^2 a_{03}^2 - 6 a_{30} a_{21} a_{12} a_{03} + 4 a_{30} a_{12}^3 + 4 a_{21}^3 a_{03} - 3 a_{21}^2 a_{12}^2)\right].\nonumber
\end{align}
If $a_{30}\ne0$, we can regard $DC_{f}(\gamma_3)$ as a quadratic with respect to $\gamma_3$, and the discriminant of $DC_{f}(\gamma_3)$ as the quadratic with respect to $\gamma_3$ is $Q_{D_4}^3/144$, where
\begin{equation}\label{eq:qd4def}
Q_{D_4}:=a_{21}^2-a_{12}a_{30}.
\end{equation}
\begin{rem}
\label{rem:SignOfQD4}
If $Q_{D_4}>0$ then $DC_{f}(\gamma_3)$ is positive, zero or negative. Since the coefficient of $\gamma_3^2$ of $DC_f(\gamma_3)$ is negative, if $Q_{D_4}=0$ then $DC_{f}(\gamma_3)$ is negative or zero, and if $Q_{D_4}<0$ then $DC_{f}(\gamma_3)$ is always negative. Thus, if $Q_{D_4}\ge0$, then the equation $DC_{f}(\gamma_3)=0$ is solved for $\gamma_3$ to give
\begin{equation}
\label{eq:gamma3}
\gamma_3=\frac{2a_{21}^3-3a_{12}a_{21}a_{30}+a_{03}a_{30}^2+2\varepsilon (a_{21}^2-a_{12}a_{30})^{3/2}}{a_{30}^2},
\end{equation}
where $\varepsilon=\pm1$. 
\end{rem}


\begin{lem}\label{lem:dcfpos}
Suppose $\gamma_2=a_{02}$, $a_{30}\not=0$ and $Q_{D_4}>0$. 
\begin{enumerate}
\item
If $DC_f(\gamma_3)\gtrless 0$, then $F$ has a 
$D_4^{\mp}$-singularity at $0$.
\item
If 
$DC_f(\gamma_3)=0$, then 
for any integer $k\ge4$ there exist $\gamma_j$ $(4\le j \le k)$ such that $F$ has a 
$D_{k+1}^\pm$-singularity at $0$. 
\end{enumerate}
\end{lem}

\begin{proof}[Proof of Lemma \ref{lem:dcfpos}]
It follows from $\gamma_2 = a_{02}$ that $j^2 F=0$. Furthermore, since $a_{30}\ne0$ and $Q_{D_4}>0$, by Remark~\ref{rem:SignOfQD4} we see that $DC_f(\gamma_3)$ is positive, zero or negative. It is well-known that a function-germ $g:\R^2,0 \to \R,0$ with $j^2 g=0$ has a $D_4^-$ (respectively, $D_4^+$) -singularity at $0$ if and only if the cubic discriminant of $j^3 g$ is positive (respectively, negative), which proves (1). 

If $DC_f(\gamma_3)=0$, that is, $\gamma_3$ satisfies \eqref{eq:gamma3}, then the change of the coordinate 
\begin{equation}\label{eq:change1}
(x,y)\mapsto\left(\dfrac{x-a_{21}y}{a_{30}},y\right)
\end{equation}
and multiplying with $6a_{30}^2$  yield that $j^3F$ is $\K^3$-equivalent to 
\begin{eqnarray}\label{gam3fgam}
x^3 - 3(a_{21}^2 - a_{12} a_{30}) x y^2 - 2 \varepsilon (a_{21}^2 - a_{12} a_{30})^{3/2} y^3.
\end{eqnarray}
Moreover, the change of the coordinate
\begin{equation}\label{eq:change2}
(x,y)\mapsto\left(x-\varepsilon \sqrt{Q_{D_4}}y, \varepsilon/(2\sqrt{Q_{D_4}})x+y\right)
\end{equation}
transforms $j^3F$ into $(- 27 \varepsilon \sqrt{Q_{D_4}}/4)x^2y$. 
Therefore, using the changes of coordinates \eqref{eq:change1} and \eqref{eq:change2} and multiplying with $6a_{30}^2\cdot (-4\varepsilon/(27\sqrt{Q_{D_4}}))$, we obtain 
\begin{equation}
\label{eq:Dk}
j^kF \sim_{\K^k} x^2 y + \sum_{4\le i+j \le k}\dfrac{c_{ij}}{i!j!}x^iy^j,
\end{equation}
where $c_{ij}$ are constants consisting of the coefficients $a_{ij}$ of \eqref{monge} and $\gamma_j$ of \eqref{eq:gamma}. Especially, the coefficients $c_{0j}$ of $y^j$ of \eqref{eq:Dk} are given by
\begin{equation}
\label{eq:c0j}
c_{0j} = \dfrac{\hat{c}_{0j}}{a_{30}^{j-2}}-\dfrac{8\varepsilon a_{30}^2}{9j!\sqrt{Q_{D_4}}}\gamma_j,
\end{equation}
where $\hat{c}_{0j}$ consists of $a_{il}$ $(4 \le i+l \le j)$, $a_{21}$ and $\sqrt{Q_{D_4}}$, because \eqref{eq:change1} and \eqref{eq:change2} have properties of preserving the terms of $y^j$ of \eqref{contactfcn}. Furthermore, after a suitable change of coordinates based on \eqref{eq:changeD4},
the coefficient of $y^j$ itself changes,
but the coefficient of $\gamma_j$ appearing in the coefficient of $y^j$ never changes.
Hence, by the change of coordinates based on \eqref{eq:changeD4}, 
we can regard the coefficients of $y^j$ consisting of $a_{is}$ and $\gamma_r$ as linear expressions $d_{0j}(\gamma_j)$ in $\gamma_j$, and thus equations $d_{0j}(\gamma_j)=0$ can be solved for $\gamma_j$. As a consequence, (2) of Proposition~\ref{prop:coordinate} yields that for any $k\ge 4$ there exist $\gamma_j$ $(4 \le j \le k)$ such that $F$ has a $D_{k+1}^\pm$-singularity, and the proof is complete. 
\end{proof}

From Lemmas~\ref{lem:A2} and \ref{lem:dcfpos}, we get the stratum (i) of Table \ref{mainthmtable1}. We denote this stratum 
by $(A_2,D_4^\pm|D_{\ge5})$-stratum. Lemmas~\ref{lem:A2} and \ref{lem:dcfpos} show that whether $\gamma_2 \ne 0$ determines that $F$ has a singularity of type $A_2$ or $D_k$ at $0$, and Lemma~\ref{lem:dcfpos} shows that $F$ has a $D_4^{\pm}$-singularity at $0$ for almost all values of $\gamma_3$, while it has a $D_{\ge5}$-singularity there for only two distinct $\gamma_3$ satisfying $DC_{f}(\gamma_3)=0$. This is the reason for which we use the notation above. 

\begin{lem}\label{lem:stratumii}
If $\gamma_2=a_{02}$, $a_{30}\not=0$ and $Q_{D_4}<0$, then $F$ 
has a $D_4^{+}$-singularity at $0$ for any $\gamma_3$.
\end{lem}
\begin{proof}[Proof of Lemma~\ref{lem:stratumii}]
Remark~\ref{rem:SignOfQD4} shows that if $Q_{D_4}<0$ then $DC_f(\gamma_3)<0$. Thus, in the same way as the proof of (1) of Lemma~\ref{lem:dcfpos}, we have the proof.  
\end{proof}

From Lemmas~\ref{lem:A2} and \ref{lem:stratumii}, we obtain the stratum (ii) of Table~\ref{mainthmtable1}. In the same manner as in the previous case, this stratum 
is denoted by $(A_2,D_4^+)$-stratum.
We set
\begin{align}
Q_{E_6} :=& a
_{40}a_{21}^4-4a_{31}a_{30}a_{21}^3+6a_{22}a_{30}^2a_{21}^2
-4a_{13}a_{30}^3a_{21}+a_{04}a_{30}^4,\label{eq:qe6def}\\
Q_{E_7} :=& -a_{40}a_{21}^3+3a_{31}a_{30}a_{21}^2
-3a_{22}a_{30}^2a_{21}+a_{13}a_{30}^3,\label{eq:qe7def}\\
Q_{E_8}:=&10a_{21}^2a_{23}a_{30}^3-5a_{14}a_{21}a_{30}^4
+a_{05}a_{30}^5\nonumber\\
&\hspace{5mm}-10a_{21}^3a_{30}^2a_{32}
+5a_{21}^4a_{30}a_{41}-a_{21}^5a_{50},\label{eq:qe8def}\\
Q_{E_*}:=&
-3a_{21}^4a_{40}^2
+a_{21}^3(12a_{31}a_{40}+a_{21}a_{50})a_{30}%
\label{eq:qestdef}\\
&\hspace{5mm}-2a_{21}^2(6a_{31}^2+3a_{22}a_{40}+2a_{21}a_{41})
a_{30}^2%
\nonumber\\
&\hspace{5mm}+6a_{21}(2a_{22}a_{31}+a_{21}a_{32})a_{30}^3%
+(-3a_{22}^2-4a_{21}a_{23})
a_{30}^4%
+a_{14}a_{30}^5.\nonumber
\end{align}

\begin{lem}\label{lem:strata3_5}
Suppose that $\gamma_2=a_{02}$, $a_{30}\not=0$ and $Q_{D_4}=
0$. 
\begin{enumerate}
\item
If $DC_f(\gamma_3)\not= 0$, then $F$
has a $D_4^{+}$-singularity at $0$.
\item
If $DC_f(\gamma_3)=0$ and $\gamma_4\ne Q_{E_6}/a_{30}^4$, then $F$
has an $E_6$-singularity at $0$.
\item
If $DC_f(\gamma_3)=0$, $\gamma_4=Q_{E_6}/a_{30}^4$ and $Q_{E_7}\ne 0$, then $F$
has an $E_7$-singularity at $0$.
\item
If $DC_f(\gamma_3)=0$, $\gamma_4=Q_{E_6}/a_{30}^4$, $Q_{E_7}=0$ and $\gamma_5\not=Q_{E_8}/a_{30}^5$, then $F$
has an $E_8$-singularity at $0$.
\item
If $DC_f(\gamma_3)=0$, $\gamma_4=Q_{E_6}/a_{30}^4$, $Q_{E_7}=0$,
$\gamma_5=Q_{E_8}/a_{30}^5$ and $Q_{E_*}\not=0$, then
$j^5F\sim_{\K^5}x^3+xy^4$.
\end{enumerate}
\end{lem}

\begin{proof}[Proof of Lemma \ref{lem:strata3_5}]
By Remark~\ref{rem:SignOfQD4}, now $DC_{f}(\gamma_3)$ is negative or zero.

If $DC_f(\gamma_3)\not= 0$, then the cubic discriminant of $j^3F$ is negative. Thus, $F$ has a $D_4^{+}$-singularity at $0$, which proves (1). 

On the other hand, if $DC_{f}(\gamma_3)=0$, then it follows from the assumption $Q_{D_4}=0$ and \eqref{gam3fgam} that $j^3F\sim_{\K^3}x^3$. By (3) of Proposition~\ref{prop:coordinate} we have 
\begin{equation}
\label{eq:E6-7}
j^4F \sim_{\K^4} x^3+\dfrac{Q_{E_7}}{a_{30}^2}xy^3+\dfrac{(Q_{E_6}-a_{30}^4\gamma_4)}{4a_{30}^2}y^4.
\end{equation}
If $\gamma_4\ne Q_{E_6}/a_{30}^4$, then a suitable change of coordinates   removes the term $Q_{E_7}/a_{30}^2 \;xy^3$ of \eqref{eq:E6-7}. This proves (2) and (3).

If $DC_{f}(\gamma_3)=Q_{D_4}=Q_{E_6}-a_{30}^4\gamma_4=Q_{E_7}=0$, then it follows form (3) of Proposition~\ref{prop:coordinate} that we have
$$
j^5F \sim_{\K^5} x^3+\dfrac{Q_{E_*}}{4a_{30}^4}xy^4+\dfrac{(Q_{E_8}-a_{30}^5\gamma_5)}{20a_{30}^3}y^5.
$$
A similar argument as above proves (4) and (5).  
\end{proof}

From Lemmas~\ref{lem:A2} and \ref{lem:strata3_5}, we obtain the strata (iii) and (v) of Table \ref{mainthmtable1}. In the same manner as in the previous cases, the strata (iii) and (v) are denoted by $(A_2, D_4^+|E_6|E_7)$-stratum and $(A_2, D_4^+|E_6|E_8|E_*)$-stratum, respectively. 

Next we consider the case $a_{30}=0$.
We set
\begin{align}
R_{A_3}:=&-3a_{21}^2+a_{02}a_{40},\label{eq:ra3def}\\
R_{A_4}:=&-10a_{21}a_{31}a_{40}+5a_{12}a_{40}^2+3a_{21}^2a_{50},
\label{eq:ra4def}\\
R_{A_5}:=&
150 a_{21}a_{31}^2a_{40}^2
-225a_{21}a_{22}a_{40}^3
+25a_{03}a_{40}^4
\label{eq:ra5def}\\
&+225a_{21}^2a_{40}^2a_{41}
-180a_{21}^2a_{31}a_{40}a_{50}
+54a_{21}^3a_{5}^2
-45a_{21}^3a_{40}a_{60}\nonumber\\
R_{A_6}:=&756a_{21}^3a_{50}^3 - 315a_{21}^2a_{50}(10a_{31}a_{50}+3a_{21}a_{60})a_{40}
\label{eq:ra6def}\\
&+75a_{21} \left(56a_{31}^2a_{50}+21a_{21}a_{31}a_{60}+3a_{21}(14a_{41}a_{50}+a_{21}a_{70})\right)a_{40}^2
\nonumber\\
& -175\left(10a_{31}^3+30a_{21}a_{31}a_{41}+9a_{21}(a_{22}a_{50}+a_{21}a_{51})\right)a_{40}^3
\nonumber\\
&+2625(a_{22}a_{31}+a_{21}a_{32})a_{40}^4-875a_{13}a_{40}^5.\nonumber
\end{align}
\begin{lem}\label{lem:strata6_7-1}
Suppose that $\gamma_2 \ne a_{02}$, $a_{30}=0$ and $a_{21}a_{40}\not=0$.
\begin{enumerate}
\item
If $R_{A_{3}}-a_{40}\gamma_2 \gtrless 0$, then $F$ 
has an $A_3^\pm$-singularity at $0$.

\item
If $\gamma_2 = R_{A_3}/a_{40}$ and $R_{A_4}\not=0$, then $F$
has an $A_4$-singularity at $0$.

\item
If $\gamma_2 = R_{A_{3}}/a_{40}$, $R_{A_4} =0$ and $R_{A_5} - 25 a_{40}^4 \gamma_3 \gtrless 0$, then $F$
has an $A_5^\pm$-singularity at $0$.

\item
If $\gamma_2  = R_{A_{3}}/a_{40}$, $R_{A_4} = 0$, $\gamma_3 = R_{A_5}/(25a_{40}^4)$ and $R_{A_6}\not=0$, then $F$
has an $A_6$-singularity at $0$.
\end{enumerate}
\end{lem} 
\begin{proof}[Proof of Lemma \ref{lem:strata6_7-1}]

Since $\gamma_2 \ne a_{02}$ and $a_{30}=0$, multiplying with $2/(a_{02}-\gamma_2)$ and (1) of Proposition~\ref{prop:coordinate} give
$$
j^4F \sim_{\K^4} y^2+\dfrac{(R_{A_3}-a_{40}\gamma_2)}{12(a_{02}-\gamma_2)^2}x^4,
$$
which proves (1). 
The proofs for the rest of the assertions are given following the same procedure, so we will skip the proofs. 
\end{proof}




\begin{lem}\label{lem:strata6_7-2}
Suppose that $\gamma_2 \ne a_{02}$ and $a_{30}=0$.
\begin{enumerate}
\item
If $a_{21}\not=0$ and $a_{40}=0$, then $F$ 
has an $A_3^-$-singularity at $0$. 

\item
If $a_{21}=0$ and $a_{40}\not=0$, then $F$ 
has an $A_3^\pm$-singularity at $0$.
\end{enumerate}
\end{lem}
\begin{proof}[Proof of Lemma \ref{lem:strata6_7-2}]

If $a_{21}\not=0$ and $a_{40}=0$, then multiplying with $2/(a_{02}-\gamma_2)$ and (1) of Proposition~\ref{prop:coordinate} give
$$
j^4F \sim_{\K^4} y^2-\dfrac{a_{21}^2}{4(a_{02}-\gamma_2)^2}x^4. 
$$
Thus, the coefficient of $x^4$ is always negative, which proves (1). 

Similarly to the proof of (1), if $a_{21}=0$ and $a_{40}\not=0$, then we have 
$$
j^4F \sim_{\K^4} y^2+\dfrac{a_{40}}{12(a_{02}-\gamma_2)}x^4,
$$
which completes the proof of (2).
\end{proof}

%




It is easily seen that if $\gamma_2=a_{02}$, $a_{30}=0$ and $a_{21}\not=0$, then by \eqref{eq:disc} we have $DC_{f}(\gamma_3) = a_{21}^2(-4a_{21} \gamma_3 + S_{D_4})$, where $S_{D_4}:=4a_{03} a_{21} - 3a_{12}^2$. 
\begin{lem}
\label{lem:strata4,6}
Suppose that $\gamma_2=a_{02}$, $a_{30}=0$ and $a_{21}\not=0$.
\begin{enumerate}
\item
If $-4a_{21} \gamma_3 + S_{D_4} \gtrless 0$, then
$F$ 
has a $D_4^{\mp}$-singularity at $0$.

\item 
If $\gamma_3=S_{D_4}/(4a_{21})$, then 
for any integer $k\ge4$ there exist $\gamma_j$ $(4\le j \le k)$ such that $F$ 
has a $D_{k+1}^\pm$-singularity at $0$. 
\end{enumerate}
\end{lem}
\begin{proof}[Proof of Lemma \ref{lem:strata4,6}]
The proof of (1) is immediate (see the proof of (1) of Lemma~\ref{lem:dcfpos}). 

If $\gamma_3 = S_{D_4}/(4a_{21})$, then $j^3F(0)=y(2a_{21}x+a_{12}y)^2/(8a_{21})$. Thus, using the change of coordinates $(x,y) \mapsto ((x-a_{12}y)/(2a_{21}),y)$ and multiplying with $8a_{21}$, we obtain
\begin{equation}
\label{eq:Dk2}
j^kF\sim_{\K^k} x^2y + \sum_{4\le i+j \le k}\dfrac{c_{ij}}{i!j!}x^iy^j,
\end{equation}
where $c_{ij}$ are constants consisting of the coefficients $a_{ij}$ of \eqref{monge} and $\gamma_j$ of \eqref{eq:gamma}. Especially, the coefficients $c_{0j}$ of $y^j$ of \eqref{eq:Dk2} are given by
\begin{equation*}
c_{0j} = \dfrac{\hat{c}_{0j}}{a_{21}^{j-1}}-\dfrac{8a_{21}}{j!}\gamma_j,
\end{equation*}
where $\hat{c}_{0j}$ consists of $a_{il}$ $(4 \leq i+j \leq j)$, $a_{21}$ and $a_{12}$. The rest of the proof is given by the same argument of the proof of (2) of Lemma~\ref{lem:dcfpos}. 
\end{proof}

\begin{lem}\label{proDEonridge}
Suppose that $\gamma_2=a_{02}$, $\gamma_3=a_{03}$, $a_{30}=a_{21}=0$ and $a_{12}\not=0$. 
If $a_{40}\not=0$, then $F$ 
has a $D_{5}^\pm$-singularity at $0$.
\end{lem}
\begin{proof}[Proof of Lemma \ref{proDEonridge}]
By the assumptions and \eqref{contactfcn}, it is easy to check that $j^3F= (a_{21}/2)xy^2$ holds.
Furthermore, multiplying with $2/a_{21}$ and using (2) of Proposition~\ref{prop:coordinate}, we obtain
\begin{eqnarray*}
j^4F\sim_{\K^4}
xy^2+\dfrac{a_{40}}{12a_{12}}x^4,
\end{eqnarray*}
which completes the proof. 
\end{proof}

From Lemmas~\ref{lem:strata6_7-1} and \ref{lem:strata4,6}, we obtain the strata (iv) and (vi) of Table~\ref{mainthmtable1}. From Lemmas~\ref{lem:strata6_7-2} and \ref{lem:strata4,6}, we obtain the stratum (vii) of Table~\ref{mainthmtable1}. Moreover, from Lemmas~\ref{lem:strata6_7-2} and \ref{proDEonridge}, we obtain the stratum (viii) of Table~\ref{mainthmtable1}. In the same manner as in the previous cases, the strata (iv), (vi), (vii) and (viii) are denoted by $(A_3|A_4,D_4|D_{\ge5})$-stratum, $(A_3|A_5|A_6,D_4|D_{\ge5})$-stratum, $(A_3^-,D_4|D_{\ge5})$-stratum and $(A_3,D_5)$-stratum, respectively. 

Summing up the lemmas above, we now obtain Theorem \ref{thm:main}.
\end{proof}

\begin{rem}\label{reminv}
We get several kinds of contact cylindrical surfaces in the above proof,
and their existence and suitable jets are invariant under  Euclidean motions of surfaces.
For example, for a parabolic surface germ belonging to the stratum of No. $(vi)$ 
i.e. $(A_3^\pm|A_5^\pm|A_6,D_4^\pm)$-stratum,
there exists an $A_6$-contact cylindrical surface
whose $3$-jets of the base curves are uniquely given by
$$
j^3\gamma=\frac{R_{A_3}}{2a_{40}}y^2+\frac{R_{A_5}}{150a_{40}}y^3
$$
(see Lemma \ref{lem:strata6_7-1}).
Thus the notion of contact cylindrical surface gives new differential geometrical invariants. 
\end{rem}

\section{$\A$-singularity of orthogonal projection and contact cylindrical surface}\label{singoforthoangauss}
In this section, we investigate an $\A$-singularity of
the orthogonal projection of a parabolic surface-germ along the asymptotic direction,
which gives a new stratification of parabolic surface-germs. 
Then we compare such a stratification 
with the other stratification in Table \ref{mainthmtable1}
which is induced from types of contact cylindrical surfaces.
To sum up,  we have Theorem \ref{characterizationthm} and Table \ref{cor-comp-all}.
Throughout this section, we always consider a parabolic surface germ $S$ 
expressed as in (\ref{monge}).

Let $\pi_v:\R^3\to v^\perp (\subset \R^3)$ be the orthogonal projection
with the kernel direction $\bv\in S^2$.
We call $\pi_v|S: S,0 \to \bv^\perp$ {\it the orthogonal projection of $S$ along $\bv$}. 
The orthogonal projection of our surface-germ as in \eqref{monge} from the asymptotic direction $\bv=(1,0,0)$ into $\bv^\perp$ (the $yz$-plane in this case)  can be expressed as
$$
\phi(x,y):=(y,f(x,y)).
$$
It is easily seen that $\phi(x,y)$ has an unstable $\A$-singularity whose singular set is singular at $0$.
Note that the orthogonal projections of a parabolic surface-germ $S$ along the other directions have no singularities or just fold-singularities at $0$, so we do not consider such cases here.
We have the following propositions.
\begin{prop}\label{prop:projection}
We assume that $a_{30}\ne0$.
\begin{enumerate}
\item\label{it:bl} 
The map-germ $\phi$ at $0$ is a beaks $($respectively, a lips$)$
if and only if
$Q_{D_4}>0$ $($respectively, $Q_{D_4}<0)$.
\item\label{it:goose}
The map-germ $\phi$ at $0$ is a goose
if and only if
$Q_{D_4}=0$ and $Q_{E_7}\not=0$.
\item\label{it:uggoose}
The map-germ $\phi$ at $0$ is a positive ugly goose
$($respectively, a negative ugly goose$)$
if and only if
$Q_{D_4}=Q_{E_7}=0$ and 
$Q_{E_*}>0$
$($respectively, $Q_{E_*}<0)$.
\end{enumerate}
\end{prop}
\begin{prop}\label{prop:projectiona30}
We assume that $a_{30}=0$.
\begin{enumerate}
\item\label{itm:a3001}
The map-germ $\phi$ at $0$ is a gulls if and only if
$a_{21}a_{40}\not=0$ and $R_{A_4}\not=0$.
\item\label{itm:a3002}
The map-germ $\phi$ at $0$ is an ugly gulls if and only if
$a_{21}a_{40}\not=0$, $R_{A_4}=0$ and $R_{A_6}\ne0$.
\item\label{itm:a3003}
The map-germ $\phi$ at $0$ is a $12$-singularity if and only if
$a_{21}\ne0$, $a_{40}=0$ and $a_{50}(a_{21} a_{60}-5 a_{31} a_{50})\ne0$.
\item\label{itm:a3004}
The map-germ $\phi$ at $0$ is a $16^+$-singularity 
$($respectively, $16^-$-singularity$)$
if and only if
$a_{21}=0$, $a_{12}a_{40}\ne0$ and
$10 a_{31}^2-10 a_{22} a_{40}+a_{12} a_{50}>0$,
$($respectively, $10 a_{31}^2-10 a_{22} a_{40}+a_{12} a_{50}<0)$.
\end{enumerate}
\end{prop}
To prove these propositions, we prepare the following lemma
which gives the changes of coefficients of jets
by certain coordinate changes.
\begin{lemma}\label{lem:ellinochange}
{\rm (1)}
Let $\phi=(y,a_{30}x^3/3!+\sum_{i+j=4,5}a_{ij}x^iy^j/(i!j!))$, $(a_{30}\ne0)$.
Then $j^4\phi$ is $\A^4$-equivalent to
$(y,a_{30}x^3/3!+a_{13}xy^3/3!)$.
If $a_{13}=0$, then $j^5\phi$ is $\A^5$-equivalent to
$$
\left(
y,\dfrac{a_{30}}{3!}x^3
-\dfrac{3 a_{22}^2-a_{30}a_{14}}{24 a_{30}}xy^4\right).$$
{\rm (2)}
Let 
$\phi=(
y,a_{21}x^2y/2+a_{40}x^4/4!
+\sum_{i+j=5,6,7}a_{ij}x^iy^j/(i!j!))$,
$(a_{21}a_{40}\ne0)$.
Then $j^5\phi$ is $\A^5$-equivalent to
$(y,a_{21}x^2y/2+a_{40}x^4/4!+a_{50}x^5/5!)$.
If $a_{50}=0$, 
then $j^7\phi$ is $\A^7$-equivalent to
$(y,a_{21}x^2y/2+a_{40}x^4/4!+\overline{a_{70}}x^7/7!)$,
where
$$
\overline{a_{70}}=
\dfrac{35 a_{32} a_{40}^2 - 21 a_{21} a_{40} a_{51} + 3 a_{21}^2 a_{70}}
{15120 a_{21}^2}.
$$
{\rm (3)}
Let 
$
\phi=(
y,a_{21}x^2y/2
+\sum_{i+j=5,6}a_{ij}x^iy^j/(i!j!))$, $(a_{21}\ne0)$.
Then $j^6\phi$ is 
$\A^6$-equivalent to
$(y,a_{21}x^2y/2+a_{50}x^5/5!+a_{60}x^6/6!)$.\\
{\rm (4)}
Let 
$
\phi=(
y,a_{12}xy^2/2
+\sum_{i+j=4,5}a_{ij}x^iy^j/(i!j!))$, $(a_{12}\ne0, a_{31}=0)$.
Then $j^5\phi$ is 
$\A^5$-equivalent to
$$\left
(y,\dfrac{a_{12}}{2}xy^2
+\dfrac{a_{40}}{4!}x^4+
\dfrac{-10 a_{22} a_{40} + a_{12} a_{50}}{5!a_{12}}x^5\right).
$$
\end{lemma}
\begin{proof}
See
\cite[Proposition 3.6 (2)]{kabata} for a proof of (1) and
\cite[Proposition 3.5 (2)]{kabata} for a proof of (2).
Although some coefficients are
assumed to be $1$ in the proofs in \cite{kabata}, 
the same proofs work and we obtain the assertions.
We show (3).
Replacing $x$ by 
$$
x - \dfrac{1}{6}\left(\dfrac{a_{41}}{4 a_{21}}x^3+\dfrac{a_{32}}{6 a_{21}}x^2y
+\dfrac{a_{23}}{6 a_{21}}xy^2+\dfrac{a_{14}}{4 a_{21}}y^3\right),
$$
then 
$j^6f=a_{21}x^2y/2+a_{50}x^5/5!+\sum_{i+j=6}a_{ij}x^iy^j/(i!j!)+h_1(y)$.
Further, replacing $x$ by 
$$
x - \dfrac{1}{24}
\left(\dfrac{a_{51}}{5 a_{21}}x^4+\dfrac{a_{42}}{8 a_{21}}x^3y
+\dfrac{a_{33}}{9 a_{21}}x^2y^2
+\dfrac{a_{24}}{8 a_{21}}xy^3+\dfrac{a_{15}}{5 a_{21}}y^4\right),
$$
then 
$j^6f=a_{21}x^2y/2+a_{50}x^5/5!+a_{60}x^6/6!+h_2(y)$.
Here, $h_1(y)$ and $h_2(y)$ are some functions.
This shows the assertion.
The assertion (4) is shown by
replacing $x$ by 
$
x-a_{22}x^2/(2a_{12})-a_{13}xy/(3a_{12})
$.
\end{proof}

\begin{proof}[Proof of Proposition {\rm \ref{prop:projection}}]
Replacing $x$ by $x-(a_{21}/a_{30})y$, we have
$$
j^3\phi=
\left(y,\dfrac{a_{30}}6x^3-\dfrac{Q_{D_4}}{2a_{30}}xy^2+h(y)\right),
$$
where $h(y)$ is a function of a variable $y$.
By a coordinate change in the target space, we can eliminate $h(y)$,
and we see the assertion \eqref{it:bl} by the $3$-$\A$-determinacy of
a beaks and lips.
Next we assume $Q_{D_4}=0$.
Then there exist a non-zero constant $s$ and a vector $(\xi, \eta)$ $(\eta\ne0)$
such that
\begin{equation}
\label{eq:vector}
\begin{pmatrix}
a_{30} & a_{21} \\ a_{21} & a_{12}
\end{pmatrix}
= s
\begin{pmatrix}
\eta^2 & -\xi \eta \\ -\xi \eta & \xi^2
\end{pmatrix}
\end{equation}
(see \cite[\S 4.1]{FH2012} for example).
Under the condition \eqref{eq:vector},
\begin{align}
j^5\phi=&
\left(y,
\dfrac{\widetilde{a_{30}}}{6}x^3
+\dfrac{\widetilde{a_{40}}}{24} x^4
+\dfrac{\widetilde{a_{31}}}{6 \eta}x^3y
+\dfrac{\widetilde{a_{22}}}{4 \eta^2}x^2y^2
+\dfrac{\widetilde{a_{13}}}{6 \eta^3}xy^3\right.
\label{eq:j5phigoo}\\
&
\left.+\dfrac{\widetilde{a_{50}}}{120}x^5
+\dfrac{\widetilde{a_{41}}}{24 \eta}x^4y
+\dfrac{\widetilde{a_{32}}}{12 \eta^2}x^3y^2
+\dfrac{\widetilde{a_{23}}}{12 \eta^3}x^2y^3
+\dfrac{\widetilde{a_{14}}}{24 \eta^4}xy^4+h(y)\right),
\nonumber
\end{align}
where $h(y)$ is a function of a variable $y$, and
\begin{align}
\widetilde{a_{30}}=&\eta^2 s,\label{eq:gooat30}\\
\widetilde{a_{22}}=&a_{40} \xi^2+2 a_{31} \eta \xi+a_{22} \eta^2,
\label{eq:gooat22}\\
\widetilde{a_{13}}=&a_{40} \xi^3+3 a_{31} \eta \xi^2
+3 a_{22} \eta^2 \xi+a_{13} \eta^3,\label{eq:gooat13}\\
\widetilde{a_{14}}=&
a_{50} \xi^4+4 a_{41} \eta \xi^3+6 a_{32} \eta^2 \xi^2
+4 a_{23} \eta^3 \xi+a_{14} \eta^4.\label{eq:gooat14}
\end{align}
We remark that the terms $\widetilde{a_{ij}}$ $(ij=31,13,50,41,32,23)$ will
not be used in the later calculations.
By \eqref{eq:vector} and \eqref{eq:gooat13}, 
we see $\widetilde{a_{13}}\ne0$ is equivalent to 
$Q_{E_7}\ne0$. 
Since a goose is $4$-determined,
this shows the assertion \eqref{it:goose}.
If $\widetilde{a_{13}}=0$,
then since $\eta\ne0$, we have
$
a_{13}=-(a_{40}\xi^3+3a_{31}\xi^2\eta+3a_{22}\xi\eta^2)/\eta^3
$.
Substituting this into
\eqref{eq:gooat30}, \eqref{eq:gooat22}, \eqref{eq:gooat14},
and by \eqref{eq:vector},
we see that $3\widetilde{a_{22}}^2-\widetilde{a_{30}}\widetilde{a_{14}}$
is a positive multiplication of 
$Q_{E_*}$.
By (1) of Lemma \ref{lem:ellinochange}, 
we have the assertion \eqref{it:uggoose}.
\end{proof}
\begin{proof}[Proof of Proposition {\rm \ref{prop:projectiona30}}]
We assume $a_{21}a_{40}\ne0$.
Then replacing $x$ by $x-a_{12}/(2a_{21})y$,
we see
$$
j^3\phi=
\left(y,\frac12 a_{21} x^2 y+h(y)\right),
$$
where $h(y)$ is a function of a variable $y$.
Further, replacing $x$ by 
$
x + c_{20} x^2 + 2 c_{11} x y + c_{02} y^2
$,
where
\begin{align*}
c_{20}=&
\dfrac{-2 a_{21} a_{31}+a_{12} a_{40}}{12 a_{21}^2},\quad
c_{11}=
-\dfrac{4 a_{21}^2 a_{22}-4 a_{12} a_{21} a_{31}+a_{12}^2 a_{40}}{32 a_{21}^3},\\
c_{02}=&\dfrac{-8 a_{13} a_{21}^3
+a_{12} (12 a_{21}^2 a_{22}-6 a_{12} a_{21} a_{31}+a_{12}^2 a_{40})}{48 a_{21}^4},
\end{align*}
then we see
$$
j^7\phi=\left(
y,
\dfrac{\alpha_{21}}{2}x^2 y
+\dfrac{\alpha_{40}}{24}x^4
+
\sum_{i+j=5,6,7}\dfrac{\alpha_{ij}}{i!j!}x^iy^j
\right).
$$
Here, $\alpha_{21}=a_{21}$, $\alpha_{40}=a_{40}$ and
\begin{align}
\alpha_{50}=&
\dfrac{-10 a_{21} a_{31} a_{40} + 5 a_{12} a_{40}^2 + 3 a_{21}^2 a_{50}}
{3 a_{21}^2},\label{eq:gulalpha50}\\
\alpha_{32}=&
\dfrac{1}{24 a_{21}^4}
\Big(
7 a_{12}^3 a_{40}^2 - 42 a_{12}^2 a_{31} a_{40} a_{21} + 
 6 a_{12} (8 a_{31}^2 + 6 a_{22} a_{40} \label{eq:gulalpha32}\\
&+ a_{12} a_{50}) a_{21}^2 - 
 8 (6 a_{22} a_{31} + a_{13} a_{40} + 3 a_{12} a_{41}) a_{21}^3 + 
 24 a_{32} a_{21}^4 \Big),\nonumber\\
\alpha_{51}=&
\dfrac{1}{48 a_{21}^5}
\Big(-25 a_{12}^3 a_{40}^3 + 150 a_{12}^2 a_{31} a_{40}^2 a_{21} 
- 5 a_{12} a_{40} (48 a_{31}^2 \label{eq:gulalpha51}\\
&+ 12 a_{22} a_{40} + 11 a_{12} a_{50}) a_{21}^2 + 
 20 (4 a_{31}^3 + 6 a_{22} a_{31} a_{40} + 4 a_{12} a_{40} a_{41} 
\nonumber\\
&+ 7 a_{12} a_{31} a_{50}) a_{21}^3 - 
 4 (40 a_{31} a_{41} + 15 a_{22} a_{50} + 6 a_{12} a_{60}) a_{21}^4 + 
 48 a_{51} a_{21}^5)\Big)\nonumber\\
\alpha_{70}=&
\dfrac{1}{72 a_{21}^6}
\Big(35 a_{12}^3 a_{40}^4 - 210 a_{12}^2 a_{31} a_{40}^3 a_{21} + 
 210 a_{12} a_{40}^2 (2 a_{31}^2 \label{eq:gulalpha70}\\
&+ a_{12} a_{50}) a_{21}^2 - 
 280 a_{31} a_{40} (a_{31}^2 + 3 a_{12} a_{50})a_{21}^3 + 
 84 (10 a_{31}^2 a_{50} \nonumber\\
&+ 3 a_{12} a_{40} a_{60}) a_{21}^4 - 504 a_{31} a_{60} a_{21}^5 + 
 72 a_{70} a_{21}^6\Big),\nonumber
\end{align}
and the other coefficients will not be used in the later calculations.
The right-hand side of \eqref{eq:gulalpha50} is proportional to
$R_{A_4}$. By Lemma \ref{lem:ellinochange}, 
we have the assertion \eqref{itm:a3001}.
We next assume $a_{21}a_{40}\ne0$ and $\alpha_{50}=0$.
Then we have
$$
a_{12}=
\dfrac{a_{21}(10 a_{31}a_{40}-3a_{21}a_{50})}{5a_{40}^2}.
$$
Substituting this into
$A=35 \alpha_{32} \alpha_{40}^2 - 21 \alpha_{21} \alpha_{40} \alpha_{51} 
+ 3 \alpha_{21}^2 \alpha_{70}$,
we see $A=0$ is equivalent to $R_{A_6}=0$.
This with Lemma \ref{lem:ellinochange}
show the assertion \eqref{itm:a3002}.
We show the assertion \eqref{itm:a3003}.
As in the above, 
replacing $x$ by $x-a_{12}/(2a_{21})y$,
we see 
$j^3\phi= \left(y, a_{21} x^2 y/2+h(y)\right)$
($h(y)$ is a function).
Further, 
replacing $x$ by 
$$
x-\dfrac{1}{2}
\left(
 \dfrac{a_{31}}{3 a_{21}} x^2
+\dfrac{a_{21} a_{22}-a_{12} a_{31}}{2 a_{21}^2}x y
+\dfrac{4 a_{13} a_{21}^2-6 a_{12} a_{21} a_{22}+3 a_{12}^2 a_{31}}
{12 a_{21}^3}y^2\right),
$$
we see 
$j^6\phi=
\left(
y,
a_{21} x^2 y/2
+\sum_{i,j=5,6}\beta_{ij}x^iy^j/(i!j!)+
h(y)
\right)$,
where $h(y)$ is a function,
$\beta_{50}=a_{50}$ and
$\beta_{60}=(a_{21} a_{60}-5 a_{31} a_{50})/a_{21}$.
This with Lemma \ref{lem:ellinochange}
shows the assertion \eqref{itm:a3003}.

We finally show the assertion \eqref{itm:a3004}.
As in the above, 
replacing $x$ by $x- a_{31} y/a_{40}$
and taking a routine coordinate change of the target,
we see 
$$
j^5\phi=
\left(y,
\dfrac{\delta_{12}}{2} xy^2+
\sum_{i,j=4,5}\dfrac{\delta_{ij}}{i!j!}x^iy^j \right),
$$
where
$$
\delta_{12}=a_{12},\ 
\delta_{22}=-
\dfrac{a_{31}^2-a_{22} a_{40}}{a_{40}},\ 
\delta_{40}=a_{40},\ 
\delta_{31}=0,\ 
\delta_{50}=a_{50}.
$$
Since 
$-10 \delta_{22} \delta_{40} + \delta_{12} \delta_{50}=
10 a_{31}^2-10 a_{22} a_{40}+a_{12} a_{50}$, and
by Lemma \ref{lem:ellinochange},
we have the assertion.
\end{proof}
\begin{rem}
Writing
$$
f(x,y) = \dfrac{a_{02}}2y^2+\sum_{l=3}f_l(x,y),\ f_l(x,y)=\sum_{i+j=l}\dfrac{a_{ij}}{i!j!}x^iy^j,
$$
the coefficients of $\xi^3,\xi^2\eta,\xi\eta^2,\eta^3$ in
$\widetilde{a_{13}}/6\eta^3$
given in \eqref{eq:gooat13} are expressed by, respectively, 
$$
-\dfrac1{6s\eta^2}\dfrac{\partial^4 f_4}{\partial x^4}(\xi,\eta),\ - \dfrac1{3s\eta^3}\dfrac{\partial^3 f_4}{\partial x^3}(\xi,\eta),\ - \dfrac1{2s\eta^4}\dfrac{\partial^2 f_4}{\partial x^2}(\xi,\eta),\ \dfrac1{\eta^3}\dfrac{\partial f_4}{\partial x}(\xi,\eta),
$$
where $(\xi,\eta)$ is a non-zero vector in \eqref{eq:vector}. Moreover, $Q_{E_7}$ can be expressed by using $(\xi,\eta)$. Theses facts imply 
that the non-zero vector $(\xi,\eta)$ plays an important role 
in the criterion of the goose. 

The condition \eqref{eq:vector} is equivalent to that there is a non-zero vector $(\tilde{\xi},\tilde{\eta})$ such that 
\begin{equation}
\label{eq:kernel}
M_3\begin{pmatrix}
\tilde{\xi} \\ \tilde{\eta}
\end{pmatrix}=\begin{pmatrix}
a_{30} & a_{21} \\ a_{21} & a_{12}
\end{pmatrix}\begin{pmatrix}
\tilde{\xi} \\ \tilde{\eta}
\end{pmatrix}=
\begin{pmatrix}
0 \\ 0 
\end{pmatrix}.
\end{equation}
The matrix $M_3$ is the expressed matrix of $(f_3)_x$, that is, 
$$
(f_3)_x = \dfrac12\begin{pmatrix}
x & y 
\end{pmatrix}
\begin{pmatrix}
a_{30} & a_{21} \\ a_{21} & a_{12}
\end{pmatrix}
\begin{pmatrix}
x \\ y
\end{pmatrix}.
$$
If follows from \eqref{eq:kernel}, that $(\tilde{\xi}, \tilde{\eta})$ 
is a kernel vector of $M_3$, and the kernel 
is generated by $(a_{21}, -a_{30})$.
We set $\tilde{K}$ by 
$$
\tilde{K} = f_{xx} f_{yy} - (f_{xy})^2 = a_{02}(a_{30}x+a_{21}y)+O(2). 
$$
The parabolic set of the surface $S$ is the zero set of $\tilde{K}$, 
and the gradient vector of $\tilde{K}$ is parallel to $(a_{30},a_{21})$
at the origin.
Hence, the kernel of $M_3$ is tangent to the parabolic set at the origin.  
\end{rem}

Here we summarize the above results.
We compare two different sorts of stratifications
given in 
Theorem \ref{thm:main}, Proposition \ref{prop:projection} and Proposition \ref{prop:projectiona30}, respectively.
We have the following:
\begin{thm}\label{characterizationthm}
A parabolic surface-germ $S$ of the form $($\ref{monge}\,$)$ belongs to one of the stratum of No. (i)--(vi) in Table \ref{cor-comp-all}
if and only if
the orthogonal projection of $S$ along the asymptotic direction has an $\A$-singularity
in the corresponding item in the ``$\A$-sing. of proj." column of Table \ref{cor-comp-all} at $0$.
\end{thm}


\begin{table}[htbp]
\centering
{
\begin{tabular}{l| l | l |c}
\hline
No. & Name & $\A$-sing. of proj. & Cod.\\
\hline
$\text{}^{\rule{0pt}{8pt}}$ 
(i)  & $(A_2,D_4|D_{\ge 5})$ &\mbox{beaks}  & 1 \\[2pt]
(ii) & $(A_2,D_4^{+})$   &\mbox{lips} & 1 \\[2pt]
(iii) &$(A_2, D_4^{+}|E_{6}|E_{7})$   & \mbox{goose} & 2 \\[2pt]
(iv) &$(A_3|A_4, D_4|D_{\ge5})$  & \mbox{gulls} & 2 \\[2pt]
(v) &$(A_2, D_4^{+}|E_{6}|E_{8}|E_*)$ & \mbox{ugly goose} & 3 \\[2pt]
(vi) &$(A_3|A_5|A_6, D_4|D_{\ge5})$ & \mbox{ugly gulls} & 3 \\[2pt]
(vii) &$(A_3^-, D_4|D_{\ge5})$ & 12-\mbox{singularity, etc.} & 3 \\[2pt]
(viii) &$(A_3, D_{5})$ & 16-\mbox{singularity, etc.} & 3 \\\hline
\end{tabular}
}
\caption{Corresponding $\A$-singularities of the orthogonal projections of parabolic surface-germs
to strata given in Table \ref{mainthmtable1}.}
\label{cor-comp-all}
\end{table}

\begin{rem}\label{genericrem}
Based on the result of transversality theorem \cite[Theorem 6.5., Definition 6.4.]{IRFT},
a surface $S\subset \R^3$ is called {\it projection generic} if any orthogonal projection of $S$ at any point $p\in S$ has an $\A$-singularity of $\A$-codimension $\le 4$.
Thus Theorem \ref{characterizationthm} gives a complete geometric characterization to the $\A$-singularities of the orthogonal projections of projection generic surfaces at parabolic points from the viewpoint of contact cylindrical surfaces ($\K$-singularities of the contact functions).
Since the singular value sets of the $\A$-singularities appearing in the orthogonal projections of projection generic surfaces are not diffeomorphic to each other,
our result implies that the information of the apparent contour could recover
the information of the contact cylindrical surface.
\end{rem}

We remark that
for a stratum of no.~(vii) (resp.~(viii)),
the $4$-jet of the orthogonal projection along the asymptotic direction is $\A^4$-equivalent to $(x,xy^2)$ (resp. $(x,x^2y+y^4)$),
whose $\A$-singularity is not uniquely determined.

\section*{Acknowledgments}
The authors thank Toshizumi Fukui and Farid Tari for helpful discussions.
This work is partially supported by JSPS KAKENHI Grant Numbers 21K03230, 20K14312,
18K03301,
Japan-Brazil bilateral project JPJSBP1 20190103
and Japan-Russia Research Cooperative Program 120194801.


\addresslist
\end{document}